\documentclass[11pt,leqno]{article}

\usepackage{amsmath,amssymb,amsthm}
\usepackage{mathrsfs}
\usepackage[all]{xy}
\usepackage{lscape}

\usepackage[hidelinks]{hyperref}
\usepackage[nameinlink]{cleveref}

\makeatletter

\newtheorem*{theorem-main}{Theorem~\ref{main.thm}}

\newtheorem{theorem}{Theorem}[section]
\newtheorem{proposition}[theorem]{Proposition}
\newtheorem{lemma}[theorem]{Lemma}
\newtheorem{corollary}[theorem]{Corollary}

\newtheorem{example}[theorem]{Example}
\newtheorem{Definition}[theorem]{Definition}

\theoremstyle{remark}

\theoremstyle{remark}
\newtheorem{remark}[theorem]{Remark}\def\({{\rm (}}
\def\){{\rm )}}

\let\Mathrm\operator@font

\let\Bbb\mathbb
\newcommand{\fm}{\ensuremath{\mathfrak m}}

\newcommand{\fp}{\ensuremath{\mathfrak p}}

\def\standop#1{\mathop{\Mathrm #1}\nolimits}
\def\difstop#1#2{\expandafter\def\csname #1\endcsname{\standop{#2}}}
\def\defstop#1{\difstop{#1}{#1}}

\defstop{AB}
\defstop{ann}
\defstop{Ass}
\defstop{Add}
\defstop{Alt}
\defstop{Ass}
\difstop{adim}{AFHdim}

\defstop{Bl}

\newcommand{\HH}[2]{\operatorname{H}_{#1}(#2)}
\newcommand{\HHH}[2]{\operatorname{H}^{#1}(#2)}

\defstop{S}
\defstop{projdim}
\defstop{flatdim}
\defstop{injdim}
\defstop{CMFI}
\defstop{codim}
\defstop{Coh}
\defstop{coht}
\defstop{Coker}
\defstop{Cone}
\defstop{Cl}
\defstop{Cox}
\defstop{cosk}
\defstop{cd}
\defstop{cmd}
\difstop{charac}{char}

\defstop{depth}
\defstop{Der}
\defstop{Div}
\defstop{div}

\defstop{EM}
\defstop{embdim}
\defstop{End}
\defstop{ev}
\defstop{Ext}

\defstop{Func}
\defstop{Fpqc}

\defstop{GD}
\defstop{Good}
\defstop{Gal}
\defstop{Grass}

\difstop{height}{ht}
\defstop{Hom}
\defstop{Hy}

\defstop{ind}

\defstop{Ker}
\defstop{Im}

\defstop{Lch}
\defstop{length}
\defstop{Lin}
\defstop{Lqc}
\defstop{lqc}
\defstop{LQ}
\defstop{LM}
\defstop{Lie}
\defstop{Loc}

\defstop{Mat}
\defstop{Max}
\defstop{Min}
\defstop{Mod}
\defstop{mod}
\defstop{Mor}
\defstop{MCM}
\defstop{Map}
\difstop{mred}{red}

\defstop{Nerve}
\defstop{NonSFR}
\defstop{NonCMFI}
\defstop{NonNor}
\defstop{Nor}

\defstop{ob}
\defstop{Ob}

\defstop{PA}
\defstop{PM}
\defstop{PR}
\defstop{Proj}
\defstop{Prin}
\defstop{Pic}

\defstop{Qch}
\defstop{qch}

\defstop{rad}
\defstop{rank}

\defstop{res}
\defstop{Reg}
\defstop{Ref}
\defstop{Rat}

\defstop{Spec}
\defstop{supp}
\defstop{Supp}
\defstop{sup}
\defstop{Sym}
\defstop{Sing}
\defstop{SFR}
\defstop{Soc}
\defstop{Sh}

\difstop{tdeg}{trans.deg}

\defstop{Tor}
\defstop{TRD}
\difstop{trace}{tr}
\defstop{tot}

\defstop{H}
\defstop{Zar}
\defstop{FL}
\defstop{add}
\defstop{Ind}
\defstop{grmod}

\difstop{summ}{sum}

\def\({{\rm(}}
\def\){{\rm)}}

\def\fm{\mathfrak{m}}
\def\fp{\mathfrak{p}}

\def\sdarrow#1{\downarrow\hbox to 0pt{\scriptsize$#1$\hss}}
\def\suarrow#1{\uparrow\hbox to 0pt{\scriptsize$#1$\hss}}
\def\ssearrow#1{\searrow\hbox to 0pt{\scriptsize$#1$\hss}}

\def\section{\@startsection{section}{1}{\z@ }%
  {-3.5ex plus -1ex minus -.2ex}{2.3ex plus .2ex}{\bf }}

\long\def\refname{\par\kern -3ex
  \begin{center}\rm R\sc{eferences}\end{center}\par\kern 
  -2ex}

\def\@seccntformat#1{\csname the#1\endcsname.\quad}

\def\@@@sect#1#2#3#4#5#6[#7]#8{%
  \ifnum #2>\c@secnumdepth 
  \def \@svsec {}\else \refstepcounter {#1}%
  \def\@svsec{}
  \fi 
  \@tempskipa #5\relax 
  \ifdim \@tempskipa >\z@ 
  \begingroup #6\relax \@hangfrom {\hskip #3\relax 
    \@svsec}{\interlinepenalty \@M #8\par }\endgroup 
  \csname #1mark\endcsname {#7}
  \else 
  \def \@svsechd {#6\hskip #3\@svsec #8\csname #1mark\endcsname {#7}}
  \fi \@xsect {#5}}

\def\@@@startsection#1#2#3#4#5#6{%
  \if@noskipsec \leavevmode \fi \par \@tempskipa #4\relax \@afterindenttrue 
  \ifdim \@tempskipa <\z@ \@tempskipa -\@tempskipa \@afterindentfalse 
  \fi \if@nobreak \everypar {}\else \addpenalty {\@secpenalty }\addvspace 
  {\@tempskipa }\fi \@ifstar {\@ssect {#3}{#4}{#5}{#6}}{\@dblarg 
    {\@@@sect {#1}{#2}{#3}{#4}{#5}{#6}}}}

\def\theparagraph{\thesection.\arabic{paragraph}}
\def\aparagraph{\@@@startsection{paragraph}{2}{\z@ }%
  {1.75ex plus .2ex minus .15ex}{-1em}{\bf(\theparagraph) } }
\def\paragraph{\@@@startsection{paragraph}{2}{\z@ }%
  {1.75ex plus .2ex minus .15ex}{-1em}{}{\bf(\theparagraph)} }

\let\c@theorem\c@paragraph

\advance\textheight 16mm
\advance\textwidth 20mm
\title{Acyclicity test of complexes modulo Serre subcategories using the residue fields}
\author{M{\sc itsuyasu} H{\sc ashimoto}\thanks{Partially supported by MEXT Promotion of Distinctive Joint Research Center Program JPMXP0723833165.}
  \and
  X{\sc i} T{\sc ang}\thanks{Corresponding author.}}
\date{}

\begin{document}

\maketitle

\begin{abstract}
 Let $R$ be a commutative noetherian ring, and let $\mathscr{S}$(resp. $\mathscr{L}$) be a  Serre(resp. localizing) subcategory of the category of $R$-modules.  If $\Bbb F$ is an unbounded complex of $R$-modules Tor-perpendicular to $\mathscr{S}$ and $d$ is an integer, then $\HH{i\geqslant d}{S\otimes_R \Bbb F}$ is in $\mathscr{L}$  for each $R$-module $S$ in $\mathscr{S}$ if and only if $\HH{i\geqslant d}{k(\fp)\otimes_R \Bbb F}$ is in $\mathscr{L}$  for each prime ideal $\fp$ such that $R/\fp$ is in $\mathscr{S}$, where $k(\fp)$ is the residue field at $\fp$. As an application,  we show that for any $R$-module $M$,  $\Tor_{i\geqslant 0}^R(k(\fp),M)$ is in $\mathscr{L}$ for each prime ideal $\fp$ such that $R/\fp$ is in $\mathscr{S}$ if and only if $\Ext^{i \geqslant 0}_R(S,M)$ is in $\mathscr{L}$ for each cyclic $R$-module $S$ in $\mathscr{S}$. We also obtain some new characterizations of regular and Gorenstein rings in the case of $\mathscr{S}$ consists of finite modules with supports in a specialization-closed subset $V(I)$ of $\Spec R$.
 \end{abstract}

\section{Introduction}
\paragraph 
Throughout this paper, $R$ denotes a commutative noetherian ring.  We consider the spectrum $\Spec R$ of prime ideals of $R$ and its maximal spectrum $\Max R$ of maximal ideals of $R$.  For $\fp\in \Spec R$,  we denote by  $k(\fp):=R_{\fp}/{\fp R_{\fp}}$ its residue field.
There is plenty of literature on testing flatness and injectivity of modules in terms of vanishing of Tor and Ext with coefficients in residue fields $k(\fp)$(see \cite{AF, CFT, CIM, H2}). A result of interest to us in this paper is the following theorem.
\begin{theorem}\label{*Hashimoto.thm}
Let $$\Bbb F: \cdots \xrightarrow{d_2}  F_{1}\xrightarrow{d_1} F_0\xrightarrow{} 0$$
be a complex of flat $R$-modules. If  $\HH{i}{k(\fp)\otimes_R \Bbb F}=0$  for all $i\geqslant 1$ and each $\fp\in \Spec R$, then 
$\HH{i}{\Bbb F}=0$  for all $i\geqslant 1$ and $\HH{0}{\Bbb F}$ is $R$-flat. In particular, $\HH{i}{M\otimes_R \Bbb F}=0$  for all $i\geqslant 1$ for every $R$-module $M$.
\end{theorem}

\paragraph  The theorem above implies that acyclicity of a flat complex can be detected by vanishing of the left derived functor of $k(\fp)\otimes_R-$.  From another point of view, it is easily seen that the residue fields are useful in testing acyclicity of a flat complex in the quotient of module category by the subcategory merely consists of zero module. Motivated by this observation, the aim of this paper is to investigate how the residue fields determine acyclicity of a complex in Serre quotient of module category. 

\paragraph 
We denote by $\Mod R$ the category of $R$-modules  and by  $\mod R$ the full subcategory of finitely generated $R$-modules.  Let $M\in \Mod R$,  $E(M)$ denotes the injective hull of $M$ and we use  $\projdim_R M$, $\injdim_R M$ and $\flatdim_R M$ to denote the projective, injective and flat dimensions of $M$ respectively.  For unexplained terminology from homological and commutative algebra we refer to \cite{BH},  \cite{EJ} and \cite{Ma} .

\paragraph 
A class $\mathscr{S}$ of $\Mod R$  is called a {\em Serre subcategory} of $\Mod R$ if it is closed under taking submodules, factor modules, extensions. A Serre subcategory $\mathscr{L}$ of $\Mod R$ which is closed under taking arbitrary direct sums is called {\em localizing}(see \cite{Kra}). There are bijections between the set of localizing subcategories of $\Mod R$, the set of Serre subcategories of $\mod R$ and the set of specialization-closed subsets of $\Spec R$(see \cite{Ga}).

\paragraph
From now on,  $\mathscr{S}$(resp.  $\mathscr{L}$) will denote a Serre(resp. localizing) subcategory of $\Mod R$.

\paragraph 
 Let $\mathscr{X}$ be a subcategory of $\Mod R$, we consider the {\em Tor-perpendicular} category with respect to $\mathscr{X}$:
$${\mathscr{X}^{\perp}}:=\{M\in\Mod R\mid\Tor^R_{\geqslant 1}(X,M)=0,  \,\,\text{for any $X\in \mathscr{X}$} \}.$$ 

\paragraph
We recall a notation in \cite{Ta},  let $S(R):=\{\fp \in \Spec R\mid R_{\fp}  \,\,\text{is not a field}\}.$ The main result of this paper is the following,  which extends \cite[Proposition 5.3F]{AF}. 

\begin{theorem}\label{*main-flat.thm} Let
$$\Bbb F: \cdots \xrightarrow{}  F_{n+1}\xrightarrow{\partial_{n+1}} F_n\xrightarrow{\partial_n}  F_{n-1}\xrightarrow{}  \cdots $$
be a complex of modules in $\mathscr{S}^{\perp}$. Then the following assertions are equivalent for any integer $d$.
  \begin{enumerate}
   \item[\rm(1)] $\HH{i}{R/\fp\otimes_R \Bbb F} \in \mathscr{L}$  for all $i\geqslant d$ and each $\fp\in \Spec R$ such that $R/\fp\in \mathscr{S}$.
  \item[\rm(2)] $\HH{i}{k(\fp)\otimes_R \Bbb F}\in \mathscr{L}$  for all $i\geqslant d$ and each $\fp\in \Spec R$ such that $R/\fp\in \mathscr{S}$.
   \item[\rm(3)]  $\HH{i}{k(\fp)\otimes_{R_{\fp}} \Bbb F_{\fp}}\in \mathscr{L}$  for all $i\geqslant d$ and each $\fp\in \Spec R$ such that $R/\fp\in \mathscr{S}$.
      \item[\rm(4)] $\HH{i}{S\otimes_R \Bbb F} \in \mathscr{L}$  for all $i\geqslant d$ and each $S\in \mathscr{S}$. 
   \end{enumerate}
In particular,  if $\HH{i}{\Bbb F}_{\fp}\in \mathscr{L}$ for all $i\geqslant d$ and each $\fp\in \Spec R\backslash S(R)$, then the above assertions are equivalent when we replace 
$\Spec R$ with $S(R)$.
\end{theorem}

\paragraph
The proof of the theorem above is given in Section 2,  and applications are presented in Sections 3 and 4. One such, discussed in \Cref{*app.thm}, is to study when $\Tor_{i\geqslant 0}^R(S,M)$ is in $\mathscr{L}$ for any $S\in \mathscr{S}$. In Section 4, for any ideal $I$ of $R$ we introduce the concepts of relatively-$I$-flat and relatively-$I$-injective modules(see  \Cref{*relative-flat-inj.def}),  which generalize flat and injective modules respectively. As an application of \Cref{*app.thm}, we give some nontrivial examples of relatively-$I$-flat and relatively-$I$-injective modules(see \Cref{*relative-flat-inj.exm}). Next we show how these two classes of modules give rise to two new homological dimensions which we call relatively-$I$-flat and relatively-$I$-injective dimension respectively.  Moreover, we get some new characterizations of regular and Gorenstein rings(see \Cref{*regular.thm}, \Cref{*Gorenstein.thm}).

\section{Main results}
\paragraph
The following lemma is needed in the proof of the main result of this paper.
\begin{lemma}\label{*cyclic.lem}
Let $F: \mathscr{S}\to \Mod R$ be a half-exact functor that commutes with direct limits. If $F(M) \in \mathscr{L}$ for any cyclic module $M$ in $\mathscr{S}$, then $F(S) \in \mathscr{L}$ for any module $S$ in $\mathscr{S}$.
 \end{lemma}
\begin{proof} 
First we prove that $F(S)$ is in $\mathscr{L}$ for any $S\in \mod R\cap \mathscr{S}$. By induction on $r$, the number of  generators of $S$.  The case for $r=1$ is trivial. When $r>1$, there is an exact sequence $0\xrightarrow{}  S'\xrightarrow{} S\xrightarrow{} S''\xrightarrow{}  0$ in  $\mathscr{S}$ such that $S'$ is generated by $r-1$ elements and $S''$ is a cyclic module. 
Because $F$ is half-exact,  we get an exact sequence $F(S')\xrightarrow{} F(S)\xrightarrow{} F(S'')$. Thus by induction $F(S')$ and $F(S'')$ are in $\mathscr{L}$. So $F(S)$  is in $\mathscr{L}$ since $\mathscr{L}$ is a Serre subcategory of $\Mod R$. For any $S\in \mathscr{S}$,  since $\mathscr{S}$ is closed under submodules, one has that $S=\varinjlim_{j\in J}S_j$ with $S_j\in \mod R\cap \mathscr{S}$.   Note that $F(S)\cong F(\varinjlim_{j\in J}S_j)\cong \varinjlim_{i\in J}F(S_j)$ and the localizing subcategory $\mathscr{L}$ is also closed under direct limits. We conclude that $F(S) $ is in $\mathscr{L}$.
\end{proof} 

\paragraph  Now we are ready to prove \Cref{*main-flat.thm} as follows. 
\begin{proof}[Proof of Theorem \ref{*main-flat.thm}]
$(2) \Rightarrow (4) $ By \Cref{*cyclic.lem}, it suffices to consider the case that $S$ is a cyclic module in $\mathscr{S}$. 

Consider the set of ideals $$\mathcal{A}:=\{I\subset R\mid R/I\in  \mathscr{S}  \,\,\text{and}  \,\,\HH{i}{R/I\otimes_R \Bbb F} \notin \mathscr{L} \,\, \text{for some $i\geqslant d$}\}.$$
If $\mathcal{A}=\emptyset$, then we are done.  Otherwise, $\mathcal{A}\neq \emptyset$ and thus there exists a maximal element $I$ in $\mathcal{A}$ since $R$ is noetherian. 

We claim that $I$ is a prime ideal of $R$. If it is not a prime ideal, then it follows from \cite[Theorem 6.1]{Ma} that there exists an associated prime $\fp$ of $R/I$. Thus we have an exact sequence $$0 \xrightarrow{} R/\fp \xrightarrow{} R/I \xrightarrow{} R/(I+(x))\xrightarrow{}  0$$ such that $\fp=(I: x)$ for some $x\in R\backslash I$. As $R/I \in \mathscr{S}$, we get that $R/\fp, R/(I+(x))\in  \mathscr{S}$. Since each module $F_i$ in the complex $\Bbb F$ is in $\mathscr{S}^{\perp}$ by assumption, we obtain an exact sequence of complexes 
$$0 \xrightarrow{} R/\fp\otimes_R \Bbb F \xrightarrow{} R/I\otimes_R \Bbb F  \xrightarrow{} R/(I+(x))\otimes_R \Bbb F \xrightarrow{}  0.$$  
Because both the ideals $\fp$ and $I+(x)$ properly contain $I$, one has $\HH{i}{R/\fp\otimes_R \Bbb F} \in \mathscr{L}$ and $\HH{i}{R/(I+(x))\otimes_R \Bbb F}\in \mathscr{L}$ for all $i\geqslant d$. It implies that $\HH{i}{R/I\otimes_R \Bbb F} \in \mathscr{L}$ for all $i\geqslant d$. This is a contradiction.

Next we fix some $i\geqslant d$ such that $\HH{i}{R/I\otimes_R \Bbb F} \notin \mathscr{L}$.  Let $I:=\fp$ and $T:=R/\fp$ for some $\fp \in \Spec R$. Let $Q$ be the field of fractions of the domain $T$. For each $x\in R\backslash \fp$, there is an exact sequence 
$$0\xrightarrow{} T\xrightarrow{x} T\xrightarrow{} T/xT\xrightarrow{} 0.$$
 As $\fp+(x)$ properly contains $\fp$ and $R/\fp \in  \mathscr{S}$,  we have $R/\(\fp+(x)) \in  \mathscr{S}$. Then $\Tor^R_{\geqslant 1}(R/(\fp+(x),F_i)=0$ for any $F_i$ in $ \Bbb F$. Also note that $T/xT\cong R/(\fp+(x))$. We have an exact sequence of complexes 
$$0 \xrightarrow{} T\otimes_R \Bbb F \xrightarrow{x} T\otimes_R \Bbb F  \xrightarrow{} R/(\fp+(x))\otimes_R \Bbb F \xrightarrow{}  0,$$ and a long exact sequence   
$$\HH{i+1}{R/(\fp+(x))\otimes_R \Bbb F}\xrightarrow{} \HH{i}{T\otimes_R \Bbb F}  \xrightarrow{x} \HH{i}{T\otimes_R \Bbb F}.$$
Since the ideal $\fp+(x)$ properly contains $\fp$ and $\fp$ is a maximal element in $\mathcal{A}$,  we get that $\HH{i+1}{R/(\fp+(x))\otimes_R \Bbb F}\in \mathscr{L}$. 
Set $H_x:=\{m\in \HH{i}{T\otimes_R \Bbb F} \mid xm=0\}$. It is easy to see that $H_x\in \mathscr{L}$ since the subcategory $\mathscr{L}$ is closed under factor modules. Consider the natural map $$ \HH{i}{T\otimes_R \Bbb F}  \xrightarrow{f}  \HH{i}{T\otimes_R \Bbb F}\otimes_TQ.$$  Then $\Ker f=\sum_x H_x\in \mathscr{L}$. Because $ \HH{i}{T\otimes_R \Bbb F}\otimes_TQ\cong  \HH{i}{Q\otimes_TT\otimes_R \Bbb F}\cong  \HH{i}{Q\otimes_R \Bbb F}\cong  \HH{i}{k(\fp)\otimes_R \Bbb F}\in \mathscr{L}$ by assumption,  $\Im f\in \mathscr{L}$ and thus $ \HH{i}{T\otimes_R \Bbb F} \cong  \HH{i}{R/I\otimes_R \Bbb F} \in \mathscr{L}$.  This is a contradiction.

$(1) \Rightarrow (2) $  Suppose $\fp\in \Spec R$ such that $R/\fp\in \mathscr{S}$. Since $k(\fp)=\varinjlim_{}R/\fp$ by \cite[Lemma 2.4]{Kra} and  $\mathscr{L}$ is closed under direct limits,  $ \HH{i}{k(\fp)\otimes_R \Bbb F}\cong \varinjlim_{}( \HH{i}{R/\fp\otimes_R \Bbb F})$ is in $\mathscr{L}$ for all $i\geqslant d$.

$(2) \Leftrightarrow (3) $ is clear from the isomorphism $k(\fp)\otimes_R \Bbb F\cong k(\fp)\otimes_{R_{\fp}} \Bbb F_{\fp}$.

$(4) \Rightarrow (1) $ is obvious.

If $\fp \in \Spec R\backslash S(R)$, then $R_{\fp}$ is a field and thus $\fp R_{\fp}=0$. For all $i\geqslant d$, we have $$\HH{i}{k(\fp)\otimes_R\Bbb F}\cong \HH{i}{R_{\fp}\otimes_R\Bbb F}\cong \HH{i}{\Bbb F}_{\fp}\in \mathscr{L}.$$ Therefore the last assertion holds.
\end{proof}

\paragraph We point out that the corollary below is a generalization of \Cref{*Hashimoto.thm}.
\begin{corollary}\label{*flat-proj.cor} Let $d$ be an integer and
$$\Bbb F: \cdots \xrightarrow{}  F_{n+1}\xrightarrow{\partial_{n+1}} F_n\xrightarrow{\partial_n}  F_{n-1}\xrightarrow{}  \cdots $$
a complex of modules in $\mathscr{S}^{\perp}$.  If $ \HH{i}{k(\fp)\otimes_R \Bbb F}=0$ or equivalently $ \HHH{i}{\Hom_R(\Bbb F, k(\fp))}=0$ for all $i\geqslant d$ and each $\fp\in \Spec R$ such that $R/\fp\in \mathscr{S}$,  then $ \HH{i}{S\otimes_R \Bbb F}=0$ for all $i\geqslant d$ and each $S\in \mathscr{S}$. Suppose in addition that $R\in \mathscr{S}$,  then $ \HH{i}{\Bbb F}=0$ for all $i\geqslant d$.
 \end{corollary}
\begin{proof}
Note that $\Hom_R(\Bbb F, k(\fp))\cong \Hom_{ k(\fp)}( k(\fp)\otimes_R \Bbb F, k(\fp))$ and $\Hom_{k(\fp)}(-, k(\fp))$ is a faithful functor.
For any integer $i$, $ \HH{i}{k(\fp)\otimes_R \Bbb F}=0$ if and only if $ \HHH{i}{\Hom_R(\Bbb F, k(\fp))}=0$. Thus \Cref{*main-flat.thm} implies that  $\HH{i}{S\otimes_R \Bbb F}=0$ for all $i\geqslant d$ and each $S\in \mathscr{S}$. The last assertion is obvious when $\mathscr{S}$ contains $R$.
\end{proof}

\begin{remark}\label{*flat-proj.rmk}  Under the assumptions of \Cref{*flat-proj.cor},  we can not get $\HHH{i}{\Hom_R(\Bbb F, S)}=0$ for all $i\geqslant d$ and each $S\in \mathscr{S}$ in general. For example, let $R=\Bbb Z$ and $\mathscr{S}=\mod \Bbb Z$, consider a complex of projective modules
    $$\Bbb P: \cdots \xrightarrow{\partial_2}  P_{1}\xrightarrow{\partial_1} P_0\xrightarrow{} 0  $$ such that $ \HH{i}{\Bbb P}=0$ for all $i\geqslant 1$ and $ \HH{0}{\Bbb P}=\Bbb Q$. Since $\Bbb Q$ is a flat $\Bbb Z$-module, all the assumptions of \Cref{*flat-proj.cor} are satisfied, but $\HHH{1}{\Hom_{\Bbb Z}(\Bbb P, \Bbb Z)}\cong \Ext^1_{\Bbb Z}(\Bbb Q, \Bbb Z)\neq 0$.
\end{remark}

\begin{example}\label{}
 Let
  \[
  \Bbb F: 0\rightarrow F^0\xrightarrow{d^0} F^1 \xrightarrow{d^1}F^2\rightarrow\cdots
  \]
  be a complex of flat $R$-modules.
  Even if $\HHH{i}{k(\fp)\otimes_R\Bbb F}=0$ for $i>0$ and $\fp\in\Spec R$, we do not have that $\Bbb F$ is acyclic in general.
  Let $(R,\fm)$ be a discrete valuation ring, and $E=E(R/\fm)$.
  If $\Bbb F=0\rightarrow F^0\xrightarrow{d^0} F^1\rightarrow 0$ is a deleted projective resolution of $E$ (shifted by degree one),
  then $\HHH{i}{\Bbb F\otimes_R k(\fp)}=0$ for any $i>0$ and $\fp\in\Spec R$, but $H^1(\Bbb F)=E\neq 0$.
  However, {\rm\cite[Proposition~III.2.1.14]{H1}} is an affirmative result on this direction.
\end{example}

\section{Containments of $\Tor^R_i(-,M)$ and $\Ext^i_R(-,M)$ in  $\mathscr{L}$}
\paragraph
For an ideal $I$ of $R$, by using the notion of regular sequences,  it is well-known that for 
a finitely generated $R$-module $M$,  $\Ext^i_R(R/I, M)=0$ for $0 \leqslant i \leqslant n$ if and only if $\Ext^i_R(N, M)=0$ for any  $N\in \mod R$ with  $\Supp N\subset V(I)$ and for any $0 \leqslant i \leqslant n$(see \cite[Theorem 16.6]{Ma}).  In parallel, by introducing the dual notion of coregular sequences, Ooishi proved that for an artinian $R$-module $M$,   $\Tor_i^R(R/I, M)=0$ for $0 \leqslant i \leqslant n$ if and only if $\Tor_i^R(N, M)=0$ for any $N\in \mod R$ with  $\Supp N\subset V(I)$ and for any $0 \leqslant i \leqslant n$(see \cite[Theorem 3.9]{O}). In this section, our goal is to extend and unify these classical results to Serre subcategories. Spectral sequences are a helpful tool in proving our main result of this section.

\begin{lemma}\label{*rad.lem}  Let $M$ be an $R$-module. Suppose that $I$ is an ideal of $R$ and $N\in \mod R$ such that $\Supp N\subset V(I)$.  Then the following assertions hold for any non-negative integer $n$.
\begin{enumerate}
  \item[\rm(1)] If  $\Ext^i_R(R/I, M)\in \mathscr{S}$ for $0 \leqslant i \leqslant n$, then $\Ext^i_R(N, M)\in \mathscr{S}$ for $0 \leqslant i \leqslant n$.
   \item[\rm(2)] If  $\Tor_i^R(R/I, M)\in \mathscr{S}$ for $0 \leqslant i \leqslant n$, then $\Tor_i^R(N, M)\in \mathscr{S}$ for $0 \leqslant i \leqslant n$.
    \end{enumerate}
\end{lemma}
\begin{proof} Since $N$ is finitely generated, it follows from \cite[Theorem 6.4]{Ma} that there exists a chain $0=N_0\subset N_1\subset \cdots \subset N_k=N$ of submodules of $N$ such that for each $i$ we have $N_j/N_{j-1}\cong R/\fp_j$ with $\fp_j \in \Supp N$.

 (1) Note that $\Hom_R(R/I, E)$ is an injective $R/I$-module for any injective $R$-module $E$ by \cite[Lemma 3.5]{Lam}.
  For each $\fp_j$, by \cite[Theorem 10.64]{Rot}, there is a spectral sequence $$E_2^{p,q} = \Ext^p_{R/I}(R/\fp_j, \Ext^q_R(R/I, M)) \Rightarrow \Ext^{n}_R(R/\fp_j, M).$$
  Being a subquotient of a finite direct sum of copies of $\Ext^q_R(R/I,M)\in\mathscr{S}$, we have that $E_2^{p,q}\in \mathscr{S}$ for $0 \leqslant q \leqslant n$ and any $p$. 
  Since $E_{\infty}^{p,q}$ is isomorphic to a subquotient of $E_2^{p,q}$, we have $E_{\infty}^{p,q}\in \mathscr{S}$ for $0 \leqslant q \leqslant n$ and any $p$.
  Hence $\Ext^i_R(R/\fp_j, M)\in \mathscr{S}$ for $0 \leqslant i \leqslant n$.
  Consequently, applying the functor $\Ext_R^i(-, M)$ to exact sequences $0\to N_{j-1}\to N_j\to N_i/N_{j-1}\to 0$ yields that $\Ext^i_R(N, M)\in \mathscr{S}$ for $0 \leqslant i \leqslant n$.

 (2) can be proved similarly by using \cite[Theorem 10.60]{Rot}.
\end{proof}

\paragraph
Let $I=(x_1,x_2,\cdots,x_n)$ be an ideal of $R$ and  $M$ a complex of $R$-modules. In the following, we use $K(I)$ to denote the Koszul complex with respect to $I=(x_1,x_2,\cdots,x_n)$. The Koszul complex $K(I)$ is a bounded complex of finite free modules.  Then we define $H_i(K(I),M):=\Tor_i^R(K(I),M)=H_i(K(I)\otimes_RM)$ and $H^i(K(I),M):=\Ext^i_R(K(I),M)=H^i(\Hom_R(K(I),M))$. 

\paragraph
The preceding lemma allows us to extend some parts of  \cite[Theorem 2.8]{AT} to modules which are not necessarily finitely generated.
\begin{proposition}\label{Koszul-coh.prop}
 Let  $M$ be an $R$-module. If $I$ and $J$ are ideals of $R$ such that $\sqrt{I}=\sqrt{J}$, then the following assertions are equivalent for any non-negative integer $n$.
  \begin{enumerate}
  \item[\rm(1)] $\Ext^i_R(R/I, M)\in \mathscr{S}$ for $0 \leqslant i \leqslant n$.
      \item[\rm(2)] $\Ext^i_R(R/J, M)\in \mathscr{S}$ for $0 \leqslant i \leqslant n$.
   \item[\rm(3)] $H^i(K(I), M)\in \mathscr{S}$ for $0 \leqslant i \leqslant n$.
     \item[\rm(4)] $H^i(K(J), M)\in \mathscr{S}$ for $0 \leqslant i \leqslant n$.
     \end{enumerate}
\end{proposition}
\begin{proof} $(1)\Leftrightarrow (2)$ follows from \Cref{*rad.lem}(1).

$(1)\Rightarrow (3)$ There is a spectral sequence $$E_2^{p,q} = \Ext^q_R( \HH{p}{K(I)},M) \Rightarrow \HHH{p+q}{K(I), M}.$$ 
Since $I$ annihilates the homology group $ \HH{p}{K(I)}$ for all $p$,  $E_2^{p,q} \in \mathscr{S}$ for $0 \leqslant q \leqslant n$ and any $p$ by  \Cref{*rad.lem}(1).  Note that $E_{\infty}^{p,q}$ is isomorphic to a subquotient of $E_2^{p,q}$. We have that $E_{\infty}^{p,q} \in \mathscr{S}$ for $0 \leqslant q \leqslant n$ and any $p$. Since $E_{\infty}^{p,q}\cong \phi^pH^n/ \phi^{p-1}H^n$ and the filtration $\{\phi^pH^n\}$ of $H^n$ is bounded,  one has that $H^i(K(I), M)\in \mathscr{S}$ for $0 \leqslant i \leqslant n$.  

$(3)\Rightarrow (1)$ The spectral sequence $$E_2^{p,q} = \Ext^q_R(\HH{p}{K(I)},M) \Rightarrow \HHH{p+q}{K(I), M}$$ shows
$E_2^{0,0} = \Ext^0_R(R/I,M)\in \mathscr{S}$.
Assume that $\Ext^i_R(R/I,M)\in \mathscr{S}$ for $0 \leqslant i \leqslant t$,
we will show that $\Ext^{t+1}_R(R/I,M)\in \mathscr{S}$. Consider the sequence
$$0=E_2^{{-2},t+2}\xrightarrow{} E_2^{0,{t+1}} \xrightarrow{}  E_2^{2,t}.$$
Since $E_2^{2,t}\in \mathscr{S}$ by induction assumption,  we see that  $E_2^{0,t+1}=\Ext^{t+1}_R(R/I,M)\in \mathscr{S}$, provided $E_3^{0,t+1}\in\mathscr{S}$.
Since $E_4^{0,t+1}\rightarrow E_3^{0,t+1}\rightarrow E_3^{3,t-1}$ is exact and $E_3^{3,t-1}\in\mathscr{S}$, $E_3^{0,t+1}\in\mathscr{S}$, provided $E_4^{0,t+1}\in\mathscr{S}$.
Continuing this way, we see that it suffices to show that $E_\infty^{0,t+1}\in \mathscr{S}$.
This follows from the fact that there is a surjection $\HHH{t+1}{K(I),M}\to E_\infty^{0,t+1}$.
The proof of $(2)\Leftrightarrow (4)$ is similar to that of $(1)\Leftrightarrow (3)$.
\end{proof}

\paragraph
The dual version of \Cref{Koszul-coh.prop} is presented as follows, but for completeness we will give its proof.
\begin{proposition}\label{Koszul-ho.prop}
 Let  $M$ be an $R$-module. If $I$ and $J$ are ideals of $R$ such that $\sqrt{I}=\sqrt{J}$, then the following assertions are equivalent for any non-negative integer $n$.
  \begin{enumerate}
  \item[\rm(1)] $\Tor_i^R(R/I, M)\in \mathscr{S}$ for $0 \leqslant i \leqslant n$.
      \item[\rm(2)] $\Tor_i^R(R/J, M)\in \mathscr{S}$ for $0 \leqslant i \leqslant n$.
   \item[\rm(3)] $H_i(K(I), M)\in \mathscr{S}$ for $0 \leqslant i \leqslant n$.
     \item[\rm(4)] $H_i(K(J), M)\in \mathscr{S}$ for $0 \leqslant i \leqslant n$.
     \end{enumerate}
\end{proposition}
\begin{proof} $(1)\Leftrightarrow (2)$ follows from \Cref{*rad.lem}(2).

$(1)\Rightarrow (3)$  It is known from \cite[Lemma 2.5]{GJT} that there is a spectral sequence $$E^2_{p,q} = \Tor_p^R(\HH{q}{K(I)},M) \Rightarrow \HH{p+q}{K(I), M}.$$ 
Since $I$ annihilates the homology group $\HH{q}{K(I)}$ for all $q$,  $E^2_{p,q} \in \mathscr{S}$ for $0 \leqslant p \leqslant n$ and any $q$ by \Cref{*rad.lem}(2).  Note that $E^{\infty}_{p,q}$ is isomorphic to a subquotient of $E^2_{p,q}$. We have that $E^{\infty}_{p,q} \in \mathscr{S}$ for $0 \leqslant p \leqslant n$ and any $q$. Since $E^{\infty}_{p,q}\cong \phi^pH_n/ \phi^{p-1}H_n$ and the filtration $\{\phi^pH_n\}$ of $H_n$ is bounded,  one has that $H_i(K(I), M)\in \mathscr{S}$ for $0 \leqslant i \leqslant n$.

$(3)\Rightarrow (1)$ The spectral sequence $$E^2_{p,q} = \Tor_p^R(\HH{q}{K(I)},M) \Rightarrow \HH{p+q}{K(I), M}$$ shows $E^2_{0,0} = \Tor_0^R(R/I,M)\in \mathscr{S}$. Assume that $\Tor_i^R(R/I,M)\in \mathscr{L}$ for $0 \leqslant i \leqslant t$,  we will show that $\Tor_{t+1}^R(R/I,M)\in \mathscr{L}$. Consider the sequence $$0=E^2_{{t+3},-1}\xrightarrow{} E^2_{{t+1},0} \xrightarrow{}  E^2_{{t-1},1}.$$ Since $E^2_{{t-1},1}\in \mathscr{S}$ by assumption,  we see that  $E^2_{t+1,0}=\Tor_{t+1}^R(R/I,M)\in \mathscr{S}$, provided $E^3_{t+1,0}\in\mathscr{S}$.
  Continuing this way, we see that it suffices to show that $E^\infty_{t+1,0}\in\mathscr{S}$.
  This is clear, since there is a surjection $\HH{t+1}{K(I),M}\rightarrow E^\infty_{t+1,0}$.

The proof of $(2)\Leftrightarrow (4)$ is similar to that of $(1)\Leftrightarrow (3)$.
\end{proof}

\paragraph
The following theorem that is the main result of this section, yields a number of equivalent characterizations
of $\Tor_{i\geqslant 0}^R( \mathscr{S},M)$ that are in $\mathscr{L}$.
\begin{theorem}\label{*app.thm}The following assertions are equivalent for any $R$-module $M$.
 \begin{enumerate}
  \item[\rm(1)] $\Tor_i^R(k(\fp),M)\in \mathscr{L}$ for any $\fp \in \Spec R$ such that $R/\fp \in \mathscr{S}$ and  any $i \geqslant 0$.
  \item[\rm(2)] $\Tor_i^R(S,M)\in \mathscr{L}$ for any $S\in \mathscr{S}$ and  any $i \geqslant 0$.
    \item[\rm(3)] $\Tor_i^R(R/I,M)\in \mathscr{L}$ for any ideal $I$ with $R/I\in \mathscr{S}$ and  any $i \geqslant 0$.
  \item[\rm(4)]  $H_i(K(I), M)\in \mathscr{L}$ for any ideal $I$ with $R/I\in \mathscr{S}$ and  any $i \geqslant 0$.
   \item[\rm(5)] $\Ext^i_R(R/I,M)\in \mathscr{L}$ for any ideal $I$ with $R/I\in \mathscr{S}$ and  any $i \geqslant 0$.
 \item[\rm(6)]  $H^i(K(I), M)\in \mathscr{L}$ for any ideal $I$ with $R/I\in \mathscr{S}$ and  any $i \geqslant 0$.
  \end{enumerate}
\end{theorem}
\begin{proof}  $(1)\Leftrightarrow (2)$ follows from \Cref{*main-flat.thm}.

$(2)\Rightarrow (3)$ is obvious.

$(3)\Rightarrow (2)$ follows from \Cref{*cyclic.lem}.

$(3)\Leftrightarrow (4)$ and $(5)\Leftrightarrow (6)$ follow from \Cref{Koszul-ho.prop} and \Cref{Koszul-coh.prop} respectively.

$(4)\Leftrightarrow (6)$  Suppose $(a_1,a_2,\cdots,a_t)$ is a system of generators of $I$. Since $H_i(K(I), M)\cong H^{t-i}(K(I), M)$  by \cite[Proposition 1.6.10]{BH}, we get the equivalence of $(4)\Leftrightarrow (6)$.
\end{proof}

\section{Characterizations of some special rings}
\paragraph
In this section, we focus on applying our previous results to new characterizations of some well-known commutative notherian rings.
\begin{Definition}\label{*relative-flat-inj.def}(\cite[Definitions 2.6.1 and 2.7.1]{SS}) Let $I$ be an ideal of $R$ and $M$ an $R$-module.
\begin{enumerate}
  \item[\rm(1)] $M$ is called {\em relatively-$I$-flat} if $\Tor^R_i(R/J,M)=0$ for all $i\geqslant 1$ and every $I$-open ideal $J$, that is to say for all ideals $J$ containing some power $I^t$ of $I$.
\item[\rm(2)] $M$ is called {\em relatively-$I$-injective} if $\Ext^R_i(R/J,M)=0$ for all $i\geqslant 1$ and every $I$-open ideal $J$.
  \end{enumerate}
\end{Definition}
\paragraph
If $I$ is a nilpotent ideal of $R$, then the class of relatively-$I$-flat (resp.  relatively-$I$-injective) $R$-modules coincides with the class of flat (resp. injective) $R$-modules. There are some non-trivial examples of  relatively-$I$-flat and relatively-$I$-injective modules.
\begin{example}\label{*relative-flat-inj.exm} 
\begin{enumerate}
  \item[\rm(1)] Let $R$ be a Noetherian domain and $I$ an ideal which is not contained in $\rad(R)$. Consider the natural map $\tau_R^I: R \xrightarrow{} \hat{R}^I$ where $\hat{R}^I:=\varprojlim R/I^t$ is the $I$-adic completion of $R$. Then $\Coker \tau_R^I$ is a non-flat relatively-$I$-flat module(see \cite[Example 2.6.2]{SS}).
    \item[\rm(2)] Let $(R, \fm, k)$ be a Gorenstein local ring with Krull dimension $d>1$. Let $\fp \in \Spec R\backslash \{\fm\}$ such that $\height(\fp)\neq 0$. Set $M:=E(R/\fp)$. Then $0<\flatdim_R M=\height(\fp)<d$ by \cite[Proposition 2.1]{Xu}. Since $M$ is injective and  $\Hom_R(k, M)=0$,  we have that $\Tor_i^R(R/J, M)=0$ for any $i \geqslant 0$ and every $\fm$-open ideal $J$ by \Cref{*app.thm}. So $M$ is a non-flat relatively-$\fm$-flat module. Next we set $M':=\Hom_R(M,M^{(X)})$ for some set $X$.  It follows from \cite[Theorem 3.2.11]{EJ} that $k\otimes_R M'\cong \Hom_R(\Hom_R(k,M),M^{(X)})=0$. Since $M'$ is a flat $R$-module, we have that $\Ext_R^i(R/J, M')=0$ for any $i \geqslant 0$ and every $\fm$-open ideal $J$ by \Cref{*app.thm} again.  On the other hand, $0<\injdim_R M'=\height(\fp)<d$ by \cite[Proposition 3.1]{Xu}, thus $M'$ is a non-injective relatively-$\fm$-injective module.
  \end{enumerate}      
\end{example}

\begin{Definition}\label{*flat-dimension.def} If $I$ is an ideal of $R$ and $M\in \Mod R$, then the {\em relatively-$I$-flat dimension} $\flatdim_R^I M  \leqslant n$ if there is an exact sequence$$0 \xrightarrow{}  F_n \xrightarrow{} \cdots  \xrightarrow{} F_1 \xrightarrow{} F_0 \xrightarrow{}  M \xrightarrow{}  0$$ such that $F_i$ is a relatively-$I$-flat module for each $i$. If no such finite resolution exists, then $\flatdim_R^I M=\infty$. Dually, we have the notion of {\em relatively-$I$-injective dimension}.
\end{Definition}

\begin{proposition}\label{*flat-dimension.prop} Let $I$ be an ideal of $R$ and $M$ an $R$-module. Then the following assertions are equivalent for any non-negative integer $n$.
\begin{enumerate}
  \item[\rm(1)]  $\flatdim_R^I M  \leqslant n$.
 \item[\rm(2)]  $\Tor_i^R(R/J,M)=0$ for all $i\geqslant n+1$ and every $I$-open ideal $J$.
  \item[\rm(3)] For any exact sequence $$0 \xrightarrow{}  K_n \xrightarrow{} F_{n-1} \xrightarrow{}  \cdots  \xrightarrow{} F_0 \xrightarrow{}  M \xrightarrow{}  0$$ such that $F_i$ is a relatively-$I$-flat module for each $i$, $K_n$ is a relatively-$I$-flat module.
   \item[\rm(4)]  $\Tor_i^R(k(\fp),M)=0$ for all $i\geqslant n+1$ and every prime ideal $\fp\in V(I)$.
    \item[\rm(5)]  $\Tor_i^R(k(\fp),M)=0$ for all $i\geqslant n+1$ and every prime ideal $\fp\in V(I)\cap S(R)$.
     \end{enumerate}
\end{proposition}
\begin{proof} It is routine to check the equivalences of $(1)\Leftrightarrow (2)\Leftrightarrow(3)$.

$(2) \Leftrightarrow (4) \Leftrightarrow (5)$ follow from \Cref{*main-flat.thm} by taking $\mathscr{S}:=\{M\in \mod R\mid  \Supp M\in V(I) \}$ and $\mathscr{L}=0$.
\end{proof}

\paragraph
 It can be easily induced from \Cref{*flat-dimension.prop} that an $R$-module $M$ is flat if and only if $\Tor_i^R(k(\fp), M)=0$ for any $\fp \in \Spec R$ or $S(R)$ and any $i \geqslant 1$.  But the following example illustrates that the result is not true when $i$ is restricted to be one. 
 
\begin{example}\label{} Let $(R, \fm, k)$ be a regular local ring with Krull dimension $d>1$. Set $M:=E(k)$.  In view of \cite[Example 14]{H2}, we have that $\Tor_1^R(k(\fp), M)=0$ for any $\fp \in \Spec R$, but $M$ is not flat.  
\end{example}

\begin{corollary}\label{*flat-rel.flat.cor}  If $I$ is an ideal of $R$ such that $I\subset \rad(R)$ and $M\in \mod R$, then the following assertions hold.
\begin{enumerate}
  \item[\rm(1)] $\flatdim_R^I M=\flatdim_R M$, $\injdim_R^I M=\injdim_R M$.
   \item[\rm(2)] $gl.dim R=\sup \{ \flatdim_R^I M\mid M\in \Mod R\}=\sup \{ \injdim_R^I M\mid M\in \Mod R\}$, where $gl.dimR$ denotes the global dimension of $R$.
   \end{enumerate}
\end{corollary}
\begin{proof}  (1) Since $\flatdim^I_R M\leqslant \flatdim_R M$ by definition, we will show that $\flatdim_R M\leqslant \flatdim_R^I M$. Assume that $\flatdim_R^I M=n$.
For any $\fm\in \Max R$, since $I\subset \rad(R)$, $\fm\in V(I)$. In view of \Cref{*flat-dimension.prop}, $\Tor_i^{R_{\fm}}(k(\fm), M_{\fm})=0$ for any $i \geqslant n+1$. As $M_{\fm}$ is a finitely generated $R_{\fm}$-module, $\flatdim_{R_{\fm}} M_{\fm}\leqslant n$. Thus the fact that $$\flatdim_R M=\sup\{ \flatdim_{R_{\fm}} M_{\fm}\mid \fm \in \Max R\}$$ means that  $\flatdim_R M\leqslant n$.  Also note that $$\injdim_R M=\sup\{ \injdim_{R_{\fm}} M_{\fm}\mid \fm \in \Max R\}$$ by \cite[Corollary 2.3]{Ba}. The other equation can be shown in a similar way.

(2) Set $m=\sup \{ \flatdim_R^I M\mid M\in \Mod R\}$ and $n=\sup \{ \injdim_R^I M\mid M\in \Mod R\}.$
It is known from \cite{Lam} that  $gl.dim R=\sup \{ \projdim_R M\mid \text{$M$ is a cyclic $R$-module}\}$
 $$\,\,\,\,\,\,\,\,\,\,\,\,\,\, \,\,\,\,\,\,\,\,\,\,\,\,\,\,\,\,\,\,\,\,\,\,\,\,\, \,\,\,\,\,\,\,\,\,\,\,\,\,\,\,\,\,\,\,\,\,\,\,\,\, \,\,\,\,\,\,\,\,\,\,\,\,\,=\sup \{ \injdim_R M\mid \text{$M$ is a cyclic $R$-module}\}.$$ As $m \leqslant $ $gl.dimR$ and $n \leqslant $ $gl.dimR$, we have  $gl.dim R=m=n$ by (1).
\end{proof}

\begin{theorem}\label{*regular.thm}  Let $I$ be an ideal of $R$ such that $I\subset \rad(R)$. Then the following assertions are equivalent.
\begin{enumerate}
  \item[\rm(1)] $R_{\fp}$ is regular for any $\fp \in \Spec R$.
   \item[\rm(2)]   $\flatdim_R^I k(\fp)< \infty$ for any $\fp \in \Spec R$.
    \item[\rm(3)]  $\flatdim_R^I k(\fm)< \infty$ for any $\fm \in \Max R$.
     \item[\rm(4)]  $\injdim_R^I k(\fp)< \infty$ for any $\fp \in \Spec R$.
      \item[\rm(5)]  $\injdim_R^I k(\fm)< \infty$ for any $\fm \in \Max R$.
     \end{enumerate}
\end{theorem}
\begin{proof} $(1)\Rightarrow (2)$ Since $R_{\fp}$ is regular, $\flatdim_R R/\fp=\projdim_R R/\fp<\infty$. Thus $\flatdim_R^I k(\fp)\leqslant \flatdim_R k(\fp)< \infty$ for any $\fp \in \Spec R$. 

$(2)\Rightarrow (3)$ is obvious.

$(3)\Rightarrow (1)$ By \Cref{*flat-rel.flat.cor} we have that $\flatdim_R k(\fm)=\flatdim_R^I k(\fm)<\infty$ for any $\fm \in \Max R$. Then $\projdim_R \fm<\infty$ and thus $R_{\fp}$ is regular for any $\fp \in \Spec R$ by \cite[Theorem 5.94]{Lam}.

$(2) \Leftrightarrow (4)$  and $(3) \Leftrightarrow (5)$ follow from the isomorphisms $$\Hom_{R_{\fp}}(\Tor_i^R(R/J,k(\fp)),E(R/\fp))\cong \Ext^i_R(R/J, \Hom_{R_{\fp}}(k(R/\fp),E(R/\fp)))$$$$\,\,\,\,\,\,\,\,\,\,\,\, \,\,\,\,\,\,\,\,\,\,\,\,\,\,\,\,\,\,\,\,\,\,\,\,\,\,\,\,\,\,\,\,\, \cong \Ext^i_R(R/J, k(R/\fp))$$ for each integer $i$, each $I$-open ideal $J$ and any $\fp \in \Spec R$.
\end{proof}

\begin{proposition}\label{*semisimple.prop} Let $I$ be an ideal of $R$ such that $I\subset \rad(R)$.  Then the following assertions are equivalent.
\begin{enumerate}
  \item[\rm(1)]  $R$ is semisimple.
 \item[\rm(2)] $gl.dim R<\infty$ and every cyclic $R$-module can be embedded into a relatively-$I$-flat module.
     \end{enumerate}
\end{proposition}
\begin{proof} $(1)\Rightarrow (2)$ is trivial.

$(2)\Rightarrow (1)$ Suppose $gl.dim R=n>0$,  then there is a cyclic $R$-module $M$ such that $\projdim_R M=n$. By assumption, there is an exact sequence $0\xrightarrow{}  M\xrightarrow{}  F\xrightarrow{}  M'\xrightarrow{}  0$ with $F$ relatively-$I$-flat. We deduce by \Cref{*flat-rel.flat.cor} that $\flatdim_R^I M'  \leqslant n$, which implies that 
$\projdim_R M= \flatdim_R^I M<n$ by \Cref{*flat-dimension.prop}. This is a contradiction. So $gl.dim R=0$ and $R$ is semisimple.
\end{proof}

\paragraph
Let $\mathcal {F}$ be a class of $R$-modules, by an
{\em $\mathcal {F}$-preenvelope} of an $R$-module $M$ we mean a
homomorphism $\varphi:M\rightarrow F$ with $F\in \mathcal {F}$ such
that for any homomorphism $f: M\rightarrow F'$ with $F'\in \mathcal
{F}$, there is a homomorphism $g: F\rightarrow F'$ such that $g\circ\varphi=f$(see \cite{EJ}).

\paragraph
Let $M$ and $N$ be $R$-modules. Since every $R$-module over a noetherian ring $R$ has a flat preenvelope by \cite[Proposition 6.5.1]{EJ}, there is a complex 
$$\Bbb F: 0\xrightarrow{}  N\xrightarrow{}  F_{0}\xrightarrow{} F_1\xrightarrow{}  \cdots $$ such that $F_i$ is flat and $\Hom_R(\Bbb F,G)$ is exact for any flat $R$-module $G$.
For $i\geqslant 0$, in the language of \cite[Section 8] {EJ},  the  $i$-th right derived functor of $-\otimes_R-$ is defined as $\Tor^i_R(M,N):=\H_i(M\otimes_R \Bbb F\textbf{.})$, where $ \Bbb F\textbf{.}$ denotes the deleted complex of  $\Bbb F$.
\begin{proposition}\label{*hereditary.prop} If $\Tor^0_R(k(\fp),J)=0$ for every ideal $J$ of $R$ and each $\fp \in \Spec R$, then $R$ is hereditary (i.e., $gl.dim R\leqslant 1$).
\end{proposition}
\begin{proof} Given an ideal $J$ of $R$,  there is monic flat preenvelope $f: J\xrightarrow{} F_0$.  Let $p:  F_0\xrightarrow{} \Coker f $ be the natural epimorphism. Then we get a flat preenvelope $g: \Coker f \xrightarrow{} F_1$.  Thus we obtain a complex $0\xrightarrow{} F_0 \xrightarrow{d} F_1$ with $d=gp$. By assumption, $k(\fp)\otimes d$ is injective for each $\fp \in \Spec R$. Then we have that $d$ is a pure injective map and $\Coker d$ is flat by \cite[Lemma 2.1.4]{H1}. Hence $J$ is also flat. Therefore $J$ is projective and $R$ is a hereditary ring.
\end{proof}

\begin{remark}\label{*hereditary.rmk} The converse of \Cref{*hereditary.prop}  is false.  For example, let $R=\Bbb Z$, then $R$ is a hereditary ring, but $\Tor^0_R(k(0),J)\neq 0$ for any non-zero ideal $J$ of $R$.
\end{remark}

 \begin{theorem}\label{*Gorenstein.thm} The following assertions are equivalent.
 \begin{enumerate}
 \item[\rm(1)]  $R_{\fp}$ is Gorenstein for any $\fp \in \Spec R$.
\item[\rm(2)]  $\Tor_{i}^R(k(\fp),E(R/\fp))=0$ for any $\fp \in \Spec R$ and any $i$ sufficiently large.
\item[\rm(3)]  $\Tor_{i}^R(k(\fm),E(R/\fm))=0$ for any $\fm \in \Max R$ and any $i$ sufficiently large.
 \item[\rm(4)] $\Ext^{i}_R(E(R/\fp), k(\fp))=0$ for any $\fp \in \Spec R$ and any $i$ sufficiently large.
 \item[\rm(5)] $\Ext^{i}_R(E(R/\fm), k(\fm))=0$ for any $\fp \in \Max R$ and any $i$ sufficiently large.
 \end{enumerate} 
\end{theorem}
\begin{proof} It is known from \cite[Proposition 2.1]{Xu} that $R_{\fp}$ is Gorenstein for any $\fp \in \Spec R$ if and only if $ \flatdim_R E(R/\fp)<\infty$ for any $\fp \in \Spec R$ if and only if $ \flatdim_R E(R/\fm)<\infty$ for any $\fm \in \Max R$.

$(1)\Rightarrow (2)$,  $(2)\Rightarrow (3)$ and $(4)\Rightarrow (5)$ are obvious.

$(3)\Rightarrow (1)$ Given $\fm \in \Max R$, by assumption there is an integer $n$ such that $\Tor_{i}^R(k(\fm),E(R/\fm))=0$ for any $i>n$. If $\fp\in \Spec R\backslash \{\fm\}$, then $\Tor_{\geqslant 0}^R( k(\fp), E(R/\fm))\cong \Tor_{\geqslant 0}^{R_\fp}( k(\fp), E(R/\fm)_{\fp})=0$. Thus \Cref{*flat-proj.cor}  implies that $\flatdim_R E(R/\fm) \leqslant n$. Hence (1) is true.

$(2)\Leftrightarrow (4)$ and $(3)\Leftrightarrow (5)$ follow from the isomorphisms $$\Ext^{i}_R(E(R/\fp), k(\fp))\cong \Ext^{i}_R(E(R/\fp), \Hom_{R_\fp}(k(\fp), E(R/\fp)))$$
$$\,\,\,\,\,\,\,\,\,\,\,\,\,\,\,\,\,\,\,\,\,\,\,\,\,\,\,\,\,\,\,\,\,\,\,\,\,\,\,\,\,\,\,\,\,\,\,\,\,\,\,\,\,\,\,\,\,\cong \Hom_{R_\fp}(\Tor_{i}^R(k(\fp), E(R/\fp)),E(R/\fp)). $$
\end{proof}

\paragraph
Recall from \cite{AB} that the codimension of  any $M\in \mod R$, denoted by $\codim_R M$,  is defined to be the least upper bound of lengths of $M$-regular sequences.
\begin{corollary}\label{*selfinjective.cor}  If $I$ is an ideal of $R$ such that $I\subset \rad(R)$, then the following  assertions are equivalent.
\begin{enumerate}
  \item[\rm(1)] $R$ is self-injective.
     \item[\rm(2)] $\codim_R R=0$  and $E(R)$ is  relatively-$I$-flat.
   \item[\rm(3)] $\codim_R R=0$ and $E(F)$ is  relatively-$I$-flat for every flat $R$-module $F$.    
       \end{enumerate}
\end{corollary}
\begin{proof} $(1)\Rightarrow (3)$ If $R$ is injective, then every injective module is flat. Thus $E(F)$ is flat.  For any $\fm\in \Max R$,  $R_{\fm}$ is injective and hence $\codim_{R_{\fm}}R_{\fm}=0$ by \cite[Theorem 3.1.17]{BH}. As $\codim_R R=\sup \{\codim_{R_{\fm}}{R_{\fm}}\mid \fm\in \Max R\}$ by \cite[Theorem 1.6]{AB},  we have that $\codim_R R=0$.

$(3)\Rightarrow (2)$ is obvious.

$(2)\Rightarrow (1)$ For any $\fm \in \Max R$, since $\codim_R R=0$, $R/\fm$ can be embedded into $R$. Then $E(R/\fm)$ is a direct summand of $E(R)$. Thus $E(R/\fm)$ is relatively-$I$-flat.  As $I\subset \rad(R)$, $\fm\in V(I)$.  It follows from \Cref{*flat-dimension.prop} that $\Tor_i^R(k(\fm), E(R/\fm))=0$ for $i\geqslant 1$.  By \Cref{*Gorenstein.thm} we have that  $R_{\fp}$ is injective for any $\fp \in \Spec R$. Therefore $R$ is injective.
\end{proof}

\begin{flushleft}
Mitsuyasu Hashimoto\\
Department of Mathematics\\
Osaka Metropolitan University\\
Sumiyoshi-ku, Osaka 558--8585, Japan\\
e-mail: {\tt mh7@omu.ac.jp}
\end{flushleft}

\begin{flushleft}
Xi Tang\\
 School of Science\\
Guilin University of Aerospace Technology\\
Guilin 541004, Guangxi Province, P. R. China\\
e-mail: {\tt tx5259@sina.com.cn}
\end{flushleft}

\end{document}